\documentclass[11pt,reqno]{amsart}
\usepackage{graphicx}
\usepackage{float}
\usepackage[caption = false]{subfig}
\usepackage{enumerate}
\usepackage{mathtools}
\usepackage{breqn}
\usepackage{array, geometry, graphicx}
\usepackage{amsmath,amsfonts,amssymb,amsthm,mathrsfs,enumitem}
\usepackage{pdfpages}
\newtheorem{theorem}{Theorem}[section]
\newtheorem{corollary}[theorem]{Corollary}
\newtheorem{lemma}[theorem]{Lemma}

\theoremstyle{definition}
\newtheorem{definition}[theorem]{Definition}
\theoremstyle{remark}
\newtheorem{remark}[theorem]{Remark}
\numberwithin{equation}{section}
\DeclareMathOperator{\RE}{Re} \DeclareMathOperator{\IM}{Im}\DeclareMathOperator{\atan}{tan^{-1}}

\begin{document}

\title{On Applications of Extended Jack's Lemma}
\thanks{Pooja Yadav is supported by The Council of Scientific and Industrial Research(CSIR). Ref.No.:08/133(0030)/2019-EMR-I}
\author{S. Sivaprasad Kumar}
\address{Department of Applied Mathematics, Delhi Technological University, Delhi--110042, India}
\email{spkumar@dce.ac.in}
\author[Pooja Yadav]{Pooja Yadav}
\address{Department of Applied Mathematics, Delhi Technological University, Delhi--110042, India}
\email{poojayv100@gmail.com}


\begin{abstract} We introduce and study the class ${\bf\mathcal{G}}(\alpha,\beta)$ comprising analytic functions associated with a sector domain, where $\alpha,\beta\in(0,1]$. Using the extended version of Jack's lemma, we deduce Open-Door lemma type sufficient conditions for functions to be in  ${\bf\mathcal{G}}(\alpha,\beta)$. Furthermore, we point out special cases of our results that align with known results.
\end{abstract}

\maketitle

\section{Introduction}

 Let $\mathbb{D}_r:=\{z\in\mathbb{C}:|z|<r\leq1\}$ and $\mathbb{D}_1=:\mathbb{D}$. For $n\in\mathbb{N}$ and $a\in\mathbb{C}$,  $\mathcal{H}[a,n]$ be the class of  functions $f$ analytic in $\mathbb{D}$ with $f(z)=a+a_{n}z^n+a_{n+1}z^{n+1}+\cdots,a_n\neq0.$ In particular, $\mathcal{H}_{1}:=\mathcal{H}[1,1]$. Let $\mathcal{P}_{\lambda}\subset\mathcal{H}_{1}$, consisting of functions $p$ such that $\RE(e^{i\lambda} p(z))>0$, for $-\pi/2<\lambda<\pi/2$.
	In particular, $\mathcal{P}_{0}=:\mathcal{P}$, the well known Carath\'{e}odory class.  Suppose $\mathcal{A}_p$ be the subclass of $\mathcal{H}[0,p]$  consisting of functions normalized by $f^{(p)}(0)=1.$ Further, let $\mathcal{S}$ be a subclass of $\mathcal{A}_1$ consisting of all univalent functions. A function $f\in\mathcal{S}$ is said to be starlike of order $\alpha$  if  $\RE(zf'(z)/f(z))>\alpha$, for some $\alpha\in[0,1)$ and the class of all such functions is denoted by $\mathcal{S}^*(\alpha)$. In particular, $\mathcal{S}^*(0)=:\mathcal{S}^*$. Further,  Brannan
and Kirvan \cite{brki} and Stankiewicz \cite{stan} independently introduced the class $\mathcal{SS}^*(\beta)$, consisting of  functions $f\in\mathcal{S}$ such that $|\arg(zf'(z))/f(z)|<\beta\pi/2$, for some $\beta\in(0,1]$ and these functions are said to be  strongly starlike functions of order $\beta$. In \cite{bucka}, Bucka and Ciozda generalized the above classes and introduced  the class   $$\mathcal{S}_{\alpha,\beta}:=\left\{f\in\mathcal{S}: -\frac{\beta\pi}{2}<\arg \dfrac{zf'(z)}{f(z)}<\frac{\alpha\pi}{2}\right\},$$ where $\alpha,\beta\in(0,2)$ with $\alpha+\beta<2$.  It holds significant connections with other important classes in Univalent function theory, prominent among them are $\mathcal{S}^*=\mathcal{S}_{1,1}$ and $\mathcal{SS}^*(\beta)=\mathcal{S}_{\beta,\beta}$. Authors in \cite{bucka} gave the extremal function and the structural formula for the function  $f\in\mathcal{S}_{\alpha,\beta}$.  Further, they derived the sharp  bounds of $|zf'(z)/f(z)|$. Later, Takahashi and Nunokawa \cite{taka}, studied the class $\mathcal{S}_{\alpha,\beta}$ for the range $\alpha,\beta\in(0,1]$ and denoted it as $\mathcal{SS}^{*}(\alpha,\beta).$ Note that $\mathcal{SS}^{*}(\min\{\alpha,\beta\})\subset\mathcal{SS}^{*}(\alpha,\beta)\subset\mathcal{SS}^{*}(\max\{\alpha,\beta\})$.  Motivated by the above,  we define the class ${\bf\mathcal{G}}(\alpha,\beta)$, as follows:

\begin{definition}
Let $\alpha,\beta\in(0,1]$ and a non-constant analytic function $h(z)=\sum_{n=0}^\infty  h_{n}z^n$   defined on $\mathbb{D}$ is said to be in the class  ${\bf\mathcal{G}}(\alpha,\beta),$ provided
\begin{equation}
     -\frac{\beta\pi}{2}<\arg h(z)<\frac{\alpha\pi}{2} .
\end{equation}
\end{definition}
For two analytic functions $f$ and $g$, $f$ is subordinate to
	$g$, written as $f\prec g$, if there exists a Schwarz function $w(z)$ such that $f(z) = g(w(z))$.  The fundamental result for the theory of  Subordination is the Jack's lemma which has been extensively studied by many authors who have studied its applications and developed various extended versions. For example, refer to \cite{dziok,sokol,taka} for more details."
  In the present investigation, by using one of its extended version, we derive sufficient conditions for functions to belong to ${\bf\mathcal{G}}(\alpha,\beta)$, leading to novel versions of the Open-Door lemma. The fundamental forms of Jack's lemma and the Open-Door lemma are outlined below:
 \begin{lemma}\cite{miller}\label{jack}
 Let $f(z)=c_nz^n+c_{n+1}z^{n+1}+c_{n+2}z^{n+2}+\cdots$  be analytic and a nonconstant function in the unit disk $\mathbb{U}$ with $f(0)=0$. If $|f(z)|$ attains its maximum value at the point $z_0$ with $|z_0|=r$, then
\begin{equation*}
\dfrac{z_0f'(z_0)}{f(z_0)}=m,\end{equation*}
where $m\geq n\geq1$ is a real number.
\end{lemma}

 For a complex number $\gamma$ with $\RE \gamma>0$ and $n\in\mathbb{N},$ we consider the
positive number $C_{n}(\gamma)=b-ic$, where
\begin{equation}\label{bc}b=\frac{n|\gamma|}{\RE \gamma}\sqrt{\frac{2\RE \gamma}{n}+1}\;\;\text{and}\;\;c=-\dfrac{n\IM{\gamma}}{\RE \gamma}.\end{equation}
In particular, $C_n(\gamma)=\sqrt{n(n+2\gamma)}$ when $\gamma$ is real. Now, for the sake of convenience, let 
$$V_1(a,b,c)= \{x+iy: x=a,\;y\le -b+c\}$$
and $$V_2(a,b,c)= \{x+iy: x=-a,\;y\ge b+c\}.$$

\begin{lemma}\cite{kuroki}\label{open door}
Let $\gamma$ be a complex number with positive real part and $n$ be an integer
with $n\ge1.$
Suppose that a function $q\in\mathcal{H}[\gamma,n]$ satisfies the condition
$$
q(z)+\frac{zq'(z)}{q(z)}\in \mathbb{C}\backslash(\bigcup_{j=1,2}V_j(0,b,c),\quad z\in\mathbb{D},
$$ where $b$ and $c$ are given in (\ref{bc}).
Then $\RE q(z)>0.$
\end{lemma}

Later, Li and Sugawa \cite{sugawa} have extended Lemma \ref{open door}, for $| \arg q| < \alpha\pi/2$ for a given $0 <\alpha\leq 1.$
 Inspired by these aforementioned works, we obtain Open-Door lemma type results by deducing sufficient conditions for  $q(z)\in{\bf\mathcal{G}}(\alpha,\beta)$, which .

\section{Preliminary}
The following results are required in proving our main results.  The extended version of Jack's lemma given by Sok\'{o}\l\;and Spelina \cite{sokol}, has been improved and stated below: 
\begin{lemma}\label{modi}
 Let $h\in\mathcal{H}[1,n]$. If there exist two points $z_{1},z_{2}\in\mathbb{D}$ such that $|z_{1}|=|z_{2}|=r$ and  for $z\in\mathbb{D}_r$,
\begin{equation}\label{eq}
  -\frac{\beta\pi}{2}=\arg h(z_{1})<\arg h(z)<\arg h(z_{2})=\frac{\alpha\pi}{2}
\end{equation}
 for some $\alpha,\beta\in(0,1]$. Assume  \begin{equation*}
\eta=\dfrac{(\alpha-\beta)\pi}{(\alpha+\beta)2}\;\; \text{and}\;\;x_{j}=(-1)^{j+1}ie^{-i\eta}(h(z_{j}))^{\frac{2}{\alpha+\beta}}.
\end{equation*}Then
 \begin{equation*}
\frac{z_{j}h'(z_{j})}{h(z_{j})}=(-1)^ji\frac{\alpha+\beta}{2}k_{j} \quad(j=1,2),
\end{equation*}
where
\begin{equation*}
\left(\frac{k_{j}}{n}+(-1)^{j+1}\tan{\eta}\right)\cos{\eta}\geq\left(x_{j}+\frac{1}{x_{j}}\right)\frac{1}{2}\;\;\text{and}\;\;x_j>0\;\; \text{for}\;\; j=1,2.
\end{equation*}
\end{lemma}
\begin{lemma}\label{lemma2}
\emph{~\cite{siva}}	For $0<\alpha\leq1$ and $-1< m<1$, then the function
	\begin{equation*}
	h(z)=\bigg(\frac{1+e^{i m\pi}z}{1-z}\bigg)^{\alpha},
	\end{equation*}
	is analytic, univalent and  convex in $\mathbb{D}$ with
	\begin{equation*}
	h(\mathbb{D})=\bigg\{w\in\mathbb{C}:
	-\alpha(1-m)\frac{\pi}{2}\leq\arg{w}\leq\alpha(1+m)\frac{\pi}{2}\bigg\}.
	\end{equation*}
\end{lemma}

\begin{lemma}\label{lemma3}
\emph{~\cite{siva}}	Let  $h\in\mathcal{H}_{1}$  and $0\leq b\leq1$ with $b+e^{im\pi}
	\neq0$, where $-1\leq m\leq1$. Also, if
	\begin{equation*}
	h(z)\prec \frac{1+e^{im\pi} z}{1-b z}\;\;\implies\,\,
	\RE	(e^{-i\lambda}h(z))>0,
	\end{equation*} where $\lambda=\atan\bigg(\dfrac{b\sin{(m\pi)}}{b\cos{(m\pi)}+1}\bigg)$.
\end{lemma}

\section{Applications of extended Jack's Lemma}
We begin with the following theorem which is an extension of Open-Door lemma:

\begin{theorem}\label{slit1}
Let $\alpha,\beta\in(0,1]$ and $h\in\mathcal{H}[1,n].$ Then $h\in{\bf\mathcal{G}}(\alpha,\beta)$, whenever
\begin{equation}\label{con}
    (h(z))^{\frac{2}{\alpha+\beta}}+\dfrac{zh'(z)}{h(z)}\in\mathbb{C}\backslash(\bigcup_{j=1,2} V_{j}(a,b,c)),\;\;(|z|<1)
    \end{equation}
where $$a=\frac{\sin{\eta}}{\sqrt{1+\frac{4\cos^2{\eta}}{(\alpha+\beta)n}}},\;\;\;\;b=\dfrac{(\alpha+\beta)n}{2\cos{\eta}}\sqrt{1+\dfrac{4\cos^2{\eta}}{(\alpha+\beta)n}},$$
\begin{equation*}
c=\dfrac{(\alpha+\beta)n\tan{\eta}}{2}\quad \text{and}\quad  \eta=\left(\frac{\alpha-\beta}{\alpha+\beta}\right)\frac{\pi}{2}.
\end{equation*}
\end{theorem}

\begin{proof}
If there exist $z_{1}$ and $z_{2}$ with $|z_{1}|=
|z_{2}|=r<1$, such that \begin{equation*}
    -\frac{\beta\pi}{2}=\arg{h(z_{1})}<\arg{h(z)}<\arg{h(z_{2})}=\frac{\alpha\pi}{2},
\end{equation*}for $|z|<r$. Then by Lemma \ref{modi}, we get
\begin{gather*}
    \frac{z_{j}h'(z_{j})}{h(z_{j})}=(-1)^ji\frac{\alpha+\beta}{2}k_{j} \quad(j=1,2),
\intertext{where for $x_{j}>0$, we have}
\frac{k_{j}}{n}\geq\left(x_{j}+\frac{1}{x_{j}}\right)\frac{\sec{\eta}}{2}+(-1)^{j}\tan{\eta}\;\;\text{and}\;\;\exp{(-i\eta)}(h(z_{j}))^{\frac{2}{\alpha+\beta}}=(-1)^jix_{j}.
\end{gather*}
For $z=z_{1}$, we have
\begin{equation*}
\RE\left(h(z_{1})^{\frac{2}{\alpha+\beta}}+\frac{z_{1}h'(z_{1})}{h(z_{1})}\right)=x_{1}\sin{\eta}
\end{equation*}
    and
\begin{align*}
    \IM\left(h(z_{1})^{\frac{2}{\alpha+\beta}}+\frac{z_{1}h'(z_{1})}{h(z_{1})}\right)&=-x_{1}\cos{\eta}-\frac{\alpha+\beta}{2}k_{1}\\
    &\leq-x_{1}\cos{\eta}-\frac{\alpha+\beta}{2}n\left(\frac{1}{2}\left(x_{1}+\frac{1}{x_{1}}\right)\sec{\eta}-\tan{\eta}\right)\\
    &=-n\frac{\alpha+\beta}{2}\cos{\eta}f(x_{1}),
\end{align*}
where
\begin{equation*}
    f(x):=\frac{2x}{(\alpha+\beta)n}+\left(x+\frac{1}{x}\right)\frac{1}{2}\sec^{2}{\eta}-\frac{\sin{\eta}}{\cos^2{\eta}}\quad (x>0).
\end{equation*}
By a calculation, we see that $f(x)$ attains its minimum at $x:=x_{1}=1/\sqrt{1+\frac{4\cos^2{\eta}}{n(\alpha+\beta)}}.$
Thus
\begin{gather*}
    \IM\left(h(z_{1})^{\frac{2}{\alpha+\beta}}+\frac{z_{1}h'(z_{1})}{h(z_{1})}\right)\leq-\dfrac{(\alpha+\beta)n}{2\cos{\eta}}\bigg(\sqrt{1+\dfrac{4\cos^2{\eta}}{(\alpha+\beta)n}}-\sin{\eta}\bigg)\intertext{and}
    \RE\left(h(z_{1})^{\frac{2}{\alpha+\beta}}+\frac{z_{1}h'(z_{1})}{h(z_{1})}\right)=\frac{\sin{\eta}}{\sqrt{1+\frac{4\cos^2{\eta}}{n(\alpha+\beta)}}},
\end{gather*}
which contradicts (\ref{con}). After going through the same methodology for $z=z_{2}$, we obtain
\begin{gather*}
    \IM\left(h(z_{2})^{\frac{2}{\alpha+\beta}}+\frac{z_{2}h'(z_{2})}{h(z_{2})}\right)\geq\dfrac{(\alpha+\beta)n}{2\cos{\eta}}\bigg(\sqrt{1+\dfrac{4\cos^2{\eta}}{(\alpha+\beta)n}}+\sin{\eta}\bigg)\intertext{and}
    \RE\left(h(z_{2})^{\frac{2}{\alpha+\beta}}+\frac{z_{2}h'(z_{2})}{h(z_{2})}\right)=-\frac{\sin{\eta}}{\sqrt{1+\frac{4\cos^2{\eta}}{n(\alpha+\beta)}}},\end{gather*}
again a contradiction, which eventually proves the result.
\end{proof}

\begin{remark}If we take $\alpha=\beta$ in Theorem \ref{slit1}, it reduces to a result of Nunokawa and Sok\'{o}\l ~\cite[Theorem 2.1]{nuno}.\end{remark}
Further, we obtain the following result from Theorem~\ref{slit1}:
\begin{corollary}\label{slit2}
Let $\mu\in[0,1)$ and $f\in\mathcal{A}.$ If \begin{equation*}
    \mu\frac{zf'(z)}{f(z)}+(1-\mu)\bigg(1+\frac{zf''(z)}{f'(z)}\bigg)=:R_{\mu}(z)\in\mathbb{C}\backslash(\bigcup_{j=1,2} V_{j}(0,C(\mu),0),
\end{equation*}
where $C(\mu):=(1-\mu)(\sqrt{1+2/(1-\mu)}).$ Then we have \begin{equation*}
    \RE\frac{zf'(z)}{f(z)}>0, \quad \forall z\in\mathbb{D}.
\end{equation*}
\end{corollary}
\begin{proof}
Letting $h(z)=(zf'(z)/f(z))^{1-\mu},$ $\alpha=\beta=1-\mu$ and $n=1$ in Theorem \ref{slit1}, we have
\begin{equation*}
    (h(z))^{\frac{2}{\alpha+\beta}}+\dfrac{zh'(z)}{h(z)}=\mu\frac{zf'(z)}{f(z)}+(1-\mu)\bigg(1+\frac{zf''(z)}{f'(z)}\bigg)\in\mathbb{C}\backslash(  -iy\cup iy),
\end{equation*} where $y$ is such that $|y|\geq(1-\mu)(\sqrt{1+2/(1-\mu)}).$ Hence the result follows from Theorem \ref{slit1}.
\end{proof}
Note that the corollary \ref{slit2}, reduces to \cite[Theorem 2.2]{nunohun} for $\beta=1$. Moreover, the corollary \ref{slit2} with $\mu=0$ the result reduces to the following special case of Open-Door lemma (see \cite{miller}).
\begin{corollary}\label{special}
Let $f\in\mathcal{A}$ and $
    1+\frac{zf''(z)}{f'(z)}$ lies in the complex plane with slits along $\RE w=0$ and $|\IM w|\geq\sqrt{3}$. Then \begin{equation*}
    \RE\frac{zf'(z)}{f(z)}>0, \quad \forall z\in\mathbb{D}.
\end{equation*}
\end{corollary}
\begin{remark}
If  $\mu_{1},\mu_{2}\in[0,1)$ with  $\mu_{1}\leq\mu_{2}$, we have $C(\mu_{1})\geq C(\mu_{2})$ and $C(0)=\sqrt{3}$. This leads to $R_{\mu_{2}}(z)\prec R_{\mu_{1}}(z)\prec R_{0}(z).$
\end{remark}

Next result gives sufficient conditions for a function to be in the class $\mathcal{P}_{-\lambda}.$

\begin{theorem}\label{thm1}
Let $\lambda\in[0,\pi/2)$ and $h\in\mathcal{H}_{1}.$ If
\begin{equation}\label{con2}
   e^{-i \lambda} h(z)+\dfrac{zh'(z)}{h(z)}\in \mathbb{C}\backslash(\bigcup_{j=1,2}V_j(0,b,0)) ,
\end{equation}where 
$ b=\sec{\lambda}\sqrt{1+2\cos{\lambda}}-\tan{\lambda}.$
Then $h\in\mathcal{P}_{-\lambda}.$
\end{theorem}

\begin{proof}
If $h\in\mathcal{P}_{-\lambda}$ doesn't hold then by Lemma \ref{lemma2} and Lemma \ref{lemma3}, we have $h(z)\nprec\frac{1+e^{i 2\lambda}z}{1-z}$. Therefore there exist $z_{1}$ and $z_{2}$ with $|z_{1}|=
|z_{2}|=r<1$, such that \begin{equation*}
    -\left(\frac{\pi}{2}-\lambda\right)=\arg{h(z_{1})}<\arg{h(z)}<\arg{h(z_{2})}=\frac{\pi}{2}+\lambda,
\end{equation*}for $|z|<r$. Now by Lemma \ref{modi}, we get
\begin{gather*}
    \frac{z_{j}h'(z_{j})}{h(z_{j})}=(-1)^jik_{j} \quad(j=1,2),
\intertext{where}
k_{j}\geq\frac{1}{2}\left(x_{j}+\frac{1}{x_{j}}\right)\sec{\lambda}+(-1)^{j}\tan{\lambda}\quad(x_{j}>0)
\intertext{and}
e^{-i \lambda}h(z_{j})=(-1)^jix_{j} \quad(j=1,2).
\end{gather*}
Now for  $z=z_{1}$, we have\begin{align*}
  e^{-i \lambda}  h(z_{1})+\dfrac{z_{1}h'(z_{1})}{h(z_{1})}
    &=-i(x_{1}+k_{1}).
\end{align*}
Here \begin{equation}\label{2}\RE\left(e^{-i \lambda}h(z_{1})+\dfrac{z_{1}h'(z_{1})}{h(z_{1})}\right)=0\end{equation} and
\begin{equation*}
    \IM\left(e^{-i \lambda}h(z_{1})+\dfrac{z_{1}h'(z_{1})}{h(z_{1})}\right)\leq -\left(x_{1}+\frac{1}{2}\left(x_{1}+\frac{1}{x_{1}}\right)\sec{\lambda}-\tan{\lambda}\right)=:-f(x_{1}).
\end{equation*} By a simple calculation, we obtain that  $f(x)$ $(x>0)$ attains its minimum at $x=\left(1+2\cos\lambda\right)^{-1/2}.$
Thus \begin{equation}\label{3}
   \IM\left(e^{-i \lambda}h(z_{1})+\dfrac{z_{1}h'(z_{1})}{h(z_{1})}\right)\leq-\left(\sec{\lambda}\sqrt{1+2\cos{\lambda}}-\tan{\lambda}\right).
\end{equation}Clearly, (\ref{2}) and (\ref{3}) contradicts the hypothesis given in (\ref{con2}).
Similarly for  $z=z_{2}$, we have
\begin{equation*}
    e^{-i \lambda} h(z_{2})+\dfrac{z_{2}h'(z_{2})}{h(z_{2})}=i(x_{2}+k_{2}).
\end{equation*}
Here \begin{equation}\label{2*}\RE\left(e^{-i \lambda}h(z_{2})+\dfrac{z_{2}h'(z_{2})}{h(z_{2})}\right)=0\end{equation} and
\begin{equation*}
    \IM\left(e^{-i \lambda}h(z_{2})+\dfrac{z_{2}h'(z_{2})}{h(z_{2})}\right)\geq x_{2}+\frac{1}{2}\left(x_{2}+\frac{1}{x_{2}}\right)\sec{\lambda}+\tan{\lambda}=:g(x_{2}).
\end{equation*} By a calculation we obtain that  $g(x)$ $(x>0)$ attains its minimum at $x=\left(1+2\cos\lambda\right)^{-1/2}.$
Thus \begin{align}\label{3*}
    \IM\left(e^{-i \lambda}h(z_{2})+\dfrac{z_{2}h'(z_{2})}{h(z_{2})}\right)&\geq\sec{\lambda}\sqrt{1+2\cos{\lambda}}+\tan{\lambda}\nonumber\\
    &\geq\sec{\lambda}\sqrt{1+2\cos{\lambda}}-\tan{\lambda}.
\end{align}Again (\ref{2*}) and (\ref{3*}) contradicts the hypothesis given in (\ref{con2}). Hence the result.
\end{proof}
 Now the following are some  corollaries of Theorem \ref{thm1}:

\begin{corollary}
Let $\alpha\in[0,1)$, $\lambda\in[0,\pi/2)$ and $p\in\mathcal{H}_{1}$. If
\begin{equation*}
    \dfrac{p(z)-\alpha}{1-\alpha}+e^{i \lambda}\dfrac{zp'(z)}{p(z)-\alpha}\prec e^{i \lambda}\dfrac{2bz}{1-z^2} ,
\end{equation*}
where $b$ is as given in Theorem \ref{thm1}, then \begin{equation*}
    \RE(e^{-i \lambda}p(z))>\alpha\cos\lambda.
\end{equation*}
\end{corollary}
\begin{proof}
Taking $h(z)=\frac{p(z)-\alpha}{1-\alpha}$ in (\ref{con2}), we obtain $$e^{-i \lambda} h(z)+\dfrac{zh'(z)}{h(z)}=e^{-i \lambda}\dfrac{p(z)-\alpha}{1-\alpha}+\dfrac{zp'(z)}{p(z)-\alpha}\in \mathbb{C}\backslash(\bigcup_{j=1,2}V_j(0,b,0)),$$ which can also be restated as\begin{equation*}
    \dfrac{p(z)-\alpha}{1-\alpha}+e^{i \lambda}\dfrac{zp'(z)}{p(z)-\alpha}\prec e^{i \lambda}\dfrac{2bz}{1-z^2}.
\end{equation*} Hence the result follows from Theorem \ref{thm1}.
\end{proof}
\begin{remark} For $f\in\mathcal{A}_{1}$, letting $h(z)=zf'(z)/f(z)$  in Theorem \ref{thm1}, we obtain that
  $$1+\frac{zf''(z)}{f'(z)}+(e^{-i \lambda}-1)\frac{zf'(z)}{f(z)}\prec\frac{2bz}{1-z^2}\implies \RE\left( e^{-i \lambda}\frac{zf'(z)}{f(z)}\right)>0.$$ And for $\lambda=0$, it reduces to Corollary \ref{special}.
\end{remark}
\begin{theorem}\label{thm3}
Let $\delta>0$, $\lambda\in[0,\pi/2)$ and $\alpha\in[0,1]$. If $f\in\mathcal{A}_p$ then $f'(z)(z/f(z))^{\alpha+1}\in\mathcal{P}_{-\lambda}$ whenever for some $\gamma>0$, we have
\begin{equation}\label{eq1}
\gamma\bigg\{f'(z)\left(\dfrac{z}{f(z)}\right)^{\alpha+1}\bigg\}+\delta\bigg\{1+\dfrac{zf''(z)}{f'(z)}-(\alpha+1)\dfrac{zf'(z)}{f(z)}+\alpha\bigg\}\in \mathbb{C}\backslash(\bigcup_{j=1,2}V_j(a,b,0))
\end{equation}
 where $a=\gamma\delta\sin{\lambda}/\sqrt{\delta(\delta+2p\gamma\cos^2\lambda)}$  and $b=\sec{\lambda}\sqrt{\delta(\delta+2p\gamma\cos^2{\lambda})}-\delta\tan{\lambda}.$
\end{theorem}

\begin{proof}
Let $h(z)=f'(z)(z/f(z))^{\alpha+1}/p,$ then $h(z)\neq0$ in $\mathbb{D}.$ Suppose it is not true, then there exists a point $z=z_0$, $|z_0|<1$ for which $h(z_0)=0$. In this case $h$ can be written as $$h(z)=(z-z_0)^m q(z)\;\;(m\in\mathbb{N}),$$ where $q(z)$ is analytic in $\mathbb{D}$ and $q(z_0)\neq0$. It also follows that
\begin{align}\label{al1}
\gamma\bigg\{f'(z)\left(\dfrac{z}{f(z)}\right)^{\alpha+1}\bigg\}&+\delta\bigg\{1+\dfrac{zf''(z)}{f'(z)}-(\alpha+1)\dfrac{zf'(z)}{f(z)}+\alpha\bigg\}\nonumber\\
&=p\gamma  h(z)+\delta\dfrac{zh'(z)}{h(z)}\nonumber\\
&=p\gamma (z-z_0)^m q(z)+\delta\left(\dfrac{m z}{z-z_0}+\dfrac{zq'(z)}{q(z)}\right).
\end{align}But the imaginary part of (\ref{al1}) can tends to infinity when $z\to z_0$ in a suitable direction, which contradicts (\ref{eq1}). Hence $h(z)\neq0$ in $\mathbb{D}$ and $h(0)=1$. Suppose $h(z)\not\in\mathcal{P}_{-\lambda}$ then the result holds on the same lines of Theorem \ref{thm1}.
\end{proof}
On similar lines we can obtain results by taking $h(z)=$ $zf'(z)/(p g(z))$, $zf'(z)/(p f^{(1-\alpha)}(z)g^\alpha(z))$ etc. in Theorem \ref{thm3}.

\begin{corollary}\label{cor1}
Let $f,g\in\mathcal{A}_{p}$ with $g\in\mathcal{S}_{p}^{*}$ and $\alpha\geq0$. If \\
  $(i)$ $\gamma\bigg\{\dfrac{zf'(z)}{g(z)}\bigg\}+\delta\bigg\{1+\dfrac{zf''(z)}{f'(z)}-\frac{zg'(z)}{g(z)}\bigg\}\in \mathbb{C}\backslash(\bigcup_{j=1,2}V_j(a,b,0))$  then $\dfrac{zf'(z)}{g(z)}\in\mathcal{P}_{-\lambda},$\\
   $(ii)$ $\gamma\bigg\{\dfrac{zf'(z)}{f^{(1-\alpha)}(z)g^\alpha(z)}\bigg\}+\delta\bigg\{1+\dfrac{zf''(z)}{f'(z)}-(1-\alpha)\dfrac{zf'(z)}{f(z)}-\alpha\dfrac{zg'(z)}{g(z)}\bigg\}\in \mathbb{C}\backslash(\bigcup_{j=1,2}V_j(a,b,0))$ then $\dfrac{zf'(z)}{f^{(1-\alpha)}(z)g^\alpha(z)}\in\mathcal{P}_{-\lambda},$ \\where $a$ and $b$ are given in Theorem \ref{thm3}.
\end{corollary}
\begin{remark} For $\lambda=0$,\\
$(i)$ if we take $\alpha=0$ with $\gamma,\delta$ and $p$ all equal to $1$, then Theorem \ref{thm3} reduces to Corollary \ref{special}.\\
$(ii)$  Corollary \ref{cor1} part $(ii)$, the result reduces to \cite[Theorem 3.13]{patel}.
\end{remark}
\begin{corollary}
Let $\delta>0$, $\lambda\in[0,\pi/2)$ and $\alpha\in[0,1]$. If $f\in\mathcal{A}_p$,
\begin{equation*}
P(z)\equiv\gamma\bigg\{f'(z)\left(\dfrac{z}{f(z)}\right)^{\alpha+1}\bigg\}+\delta\bigg\{1+\dfrac{zf''(z)}{f'(z)}-(\alpha+1)\dfrac{zf'(z)}{f(z)}+\alpha\bigg\},
\end{equation*}such that $P(z)$ satisfies one of the following:\\
\begin{itemize}
\item[(i)]$\RE P(z)>\gamma\delta\sin{\lambda}/\sqrt{\delta(\delta+2p\gamma\cos^2\lambda)}=:X,$\\
\item[(ii)] $P(z)\in R_{R}:=\{x+iy\in\mathbb{C};|x|\leq X, |y|<\sec{\lambda}\sqrt{\delta(\delta+2p\gamma\cos^2{\lambda})}-\delta\tan{\lambda}=:Y\},$\\
\item[(iii)] $|P(z)-p\gamma|<\delta+p\gamma$, (when $\lambda=0$),\\
\item[(iv)] For $c=\sqrt{|X^2-Y^2|}$,  \begin{equation*}\resizebox{.9\hsize}{!}{$P(z)\in R_{E}:=\{x+iy\in\mathbb{C};\begin{cases}\sqrt{(x-c)^2+y^2}+\sqrt{(x+c)^2+y^2}<2X, &\text{if}\; \max\{X,Y\}=X,\\
\sqrt{x^2+(y-c)^2}+\sqrt{x^2+(y+c)^2}<2Y, &\text{if} \;\max\{X,Y\}=Y\}\end{cases}$}
\end{equation*}\end{itemize}
then $f'(z)(z/f(z))^{\alpha+1}\in\mathcal{P}_{-\lambda}$.
\end{corollary}
In \cite[Theorem 3.1]{ebadian}, authors gave some sufficient condition for an analytic function $h\in\mathcal{H}_{1}$ to have argument bounds as $|\arg{h(z)}|<\beta\pi/2$, where $\beta\in(0,1)$.  Here in the following theorem, we obtain sufficient condition for $h\in\mathcal{G}(\alpha,\beta).$
\begin{theorem}\label{arg}
Let $\alpha,\beta\in(0,1)$ and $0<\gamma\leq1.$ If $h\in\mathcal{H}_{1}$ then $h\in\mathcal{G}(\alpha,\beta),$
whenever
\begin{equation}\label{th2}
    \delta_{1}\frac{\pi}{2}<\arg{h(z)}+\gamma\arg\left(1+\frac{zh'(z)}{h^2(z)}\right)<\delta_{2}\frac{\pi}{2},
\end{equation}where
\begin{equation*}
    \delta_{1}=-\left(\beta+\frac{2\gamma}{\pi}\atan\frac{(\alpha+\beta)\sin\frac{(1-\beta)\pi}{2}}{(\alpha+\beta)\cos{\frac{(1-\beta)\pi}{2}}+M_{1}(\alpha,\beta)}\right),
\end{equation*}
\begin{equation*}
    \delta_{2}=\alpha+\frac{2\gamma}{\pi}\atan\frac{(\alpha+\beta)\sin\frac{(1-\alpha)\pi}{2}}{(\alpha+\beta)\cos{\frac{(1-\alpha)\pi}{2}}+M_{2}(\alpha,\beta)},
\end{equation*}
\begin{equation*}
    M_{j}(\alpha,\beta)=N\left(\frac{(-1)^j(\alpha+\beta)}{2-\alpha-\beta}\sin{\eta}+\sqrt{\frac{2+\alpha+\beta}{2-\alpha-\beta}}\cos{\eta}\right)\quad (j=1,2),
\end{equation*}
\begin{equation*}
    N(x)=4x^{\frac{\alpha+\beta}{2}}\left(\left(x+\frac{1}{x}\right)\sec{\eta}+(-1)^{j}2\tan{\eta}\right)^{-1}\;\;\text{and}\;\;\eta=\dfrac{(\alpha-\beta)\pi}{(\alpha+\beta)2}.
\end{equation*}
\end{theorem}
\begin{proof}
If there exist $z_{1}$ and $z_{2}$ with $|z_{1}|=
|z_{2}|=r<1$, such that \begin{equation*}
    -\frac{\beta\pi}{2}=\arg{h(z_{1})}<\arg{h(z)}<\arg{h(z_{2})}=\frac{\alpha\pi}{2},
\end{equation*}for $|z|<r$. Then by Lemma \ref{modi}, we get
\begin{gather*}
    \frac{z_{j}h'(z_{j})}{h(z_{j})}=(-1)^ji\frac{\alpha+\beta}{2}k_{j}, \quad(j=1,2),
\intertext{where for $x_{j}>0$}
k_{j}\geq\left(\frac{1}{x_{j}}+x_{j}\right)\frac{\sec{\eta}}{2}+(-1)^{j}\tan{\eta}\;\;\text{and}\;\; h(z_{j})=\left((-1)^jix_{j}\exp\left(i\eta\right)\right)^{\frac{\alpha+\beta}{2}}.
\end{gather*} For $z=z_{1}$, we have
\begin{equation*}
    h(z_{1})\left(1+\dfrac{z_{1}h'(z_{1})}{h^2(z_{1})}\right)^{\gamma}=x^{\frac{\alpha+\beta}{2}}e^{-i\frac{\beta\pi}{2}}\left(1+\frac{e^{-i(1-\beta)\frac{\pi}{2}}(\alpha+\beta)k_{1}}{2x^{\frac{\alpha+\beta}{2}}}\right)^\gamma.
\end{equation*}Moreover, \begin{align*}
    \arg{h(z_{1})}+\gamma\arg{\left(1+\dfrac{z_{1}h'(z_{1})}{h^2(z_{1})}\right)}&=-\frac{\beta\pi}{2}+\gamma\atan{-\dfrac{\sin{((1-\beta)\frac{\pi}{2})}(\alpha+\beta)k_{1}}{2x_{1}^{\frac{\alpha+\beta}{2}}+\cos{((1-\beta)\frac{\pi}{2})}(\alpha+\beta)k_{1}}},\\
    &\leq-\frac{\beta\pi}{2}-\gamma\atan{\dfrac{\sin{((1-\beta)\frac{\pi}{2})}(\alpha+\beta)}{N(x_{1})+\cos{((1-\beta)\frac{\pi}{2})}(\alpha+\beta)}},
    \end{align*}
where $N(x)$ is defined in hypothesis. As $x>0$, by a calculation, we see that $N(x)$ attains its maximum at\begin{equation*}
        x:=x_{1}=\frac{-(\alpha+\beta)}{2-\alpha-\beta}\sin{\eta}+\sqrt{\frac{2+\alpha+\beta}{2-\alpha-\beta}}\cos{\eta}.
    \end{equation*}Thus \begin{equation*}
      \arg{h(z_{1})}+\gamma\arg{\left(1+\dfrac{z_{1}h'(z_{1})}{h^2(z_{1})}\right)}\leq-\frac{\beta\pi}{2}-\gamma\atan{\dfrac{\sin{((1-\beta)\frac{\pi}{2})}(\alpha+\beta)}{M_{1}(\alpha,\beta)+\cos{((1-\beta)\frac{\pi}{2})}(\alpha+\beta)}},
    \end{equation*}which contradicts (\ref{th2}). Similarly for $z=z_{2},$ we have\begin{equation*}
      \arg{h(z_{2})}+\gamma\arg{\left(1+\dfrac{z_{2}h'(z_{2})}{h^2(z_{2})}\right)}\geq\frac{\alpha\pi}{2}+\gamma\atan{\dfrac{\sin{((1-\alpha)\frac{\pi}{2})}(\alpha+\beta)}{M_{2}(\alpha,\beta)+\cos{((1-\alpha)\frac{\pi}{2})}(\alpha+\beta)}},
    \end{equation*}which again contradicts (\ref{th2}). Thus result follows.
    \end{proof}
    Further, we have the following  interesting corollaries from Theorem \ref{arg}:
\begin{corollary}
Let $\alpha,\beta\in(0,1)$ and $0<\gamma\leq1.$ Moreover, if $f\in\mathcal{A}$ satisfies \begin{equation}\label{cor}
   \dfrac{\delta_{1}\pi}{2}<(1-\gamma)\arg \left(\dfrac{zf'(z)}{f(z)}\right)+\gamma\arg\left(1+\dfrac{zf''(z)}{f'(z)}\right)<\dfrac{\delta_{2}\pi}{2},
\end{equation}where $\delta_{1}$ and $\delta_{2}$ are defined in Theorem \ref{arg}. Then $$-\dfrac{\beta\pi}{2}<\arg\dfrac{zf'(z)}{f(z)}<\dfrac{\alpha\pi}{2},$$ equivalently $$\dfrac{zf'(z)}{f(z)}\prec\left(\dfrac{1+e^{i(\frac{\alpha-\beta}{\alpha+\beta})\pi}z}{1-z}\right)^{(\alpha+\beta)/2}.$$
\end{corollary}
\begin{proof}
Taking $h(z)=zf'(z)/f(z)$ in Theorem \ref{arg}, we obtain \begin{equation*}
    \left(\dfrac{zf'(z)}{f(z)}\right)^{1-\gamma}\left(1+\dfrac{zf''(z)}{f'(z)}\right)^\gamma=h(z)\left(1+\frac{zh'(z)}{h^2(z)}\right)^{\gamma}.
\end{equation*}Now by  (\ref{th2}) and (\ref{cor}), we have
\begin{equation*}
   \dfrac{\delta_{1}\pi}{2}<(1-\gamma)\arg \left(\dfrac{zf'(z)}{f(z)}\right)+\gamma\arg\left(1+\dfrac{zf''(z)}{f'(z)}\right)=\arg{h(z)}+\gamma\arg\left(1+\frac{zh'(z)}{h^2(z)}\right)<\dfrac{\delta_{2}\pi}{2}.
\end{equation*} Thus the result follows at once from Theorem \ref{arg}.
\end{proof}
 By taking $h(z)=zf'(z)/f(z)$ and $\alpha=\beta$ in Theorem \ref{arg}, we obtain the following result of  Nunokawa and   Sok\'{o}\l \cite{nuno2} with an alternate proof:
\begin{corollary}
Let $f\in\mathcal{A}$, $\alpha\in(0,1)$ and $0<\gamma\leq1.$ Moreover, $f$ satisfies \begin{equation*}
    \bigg|\arg\left\{\left(\dfrac{zf'(z)}{f(z)}\right)^{1-\gamma}\left(1+\dfrac{zf''(z)}{f'(z)}\right)^\gamma\right\}\bigg|<\dfrac{\delta\pi}{2},
\end{equation*}where \begin{equation*}
    \delta=\alpha+\frac{2\gamma}{\pi}\atan\frac{2\alpha\sin\frac{(1-\alpha)\pi}{2}}{2\alpha\cos{\frac{(1-\alpha)\pi}{2}}+M(\alpha)},
\end{equation*}with $$M(\alpha)=\dfrac{4}{(\frac{1+\alpha}{1-\alpha})^{(1-\alpha)/2}+(\frac{1+\alpha}{1-\alpha})^{-(1+\alpha)/2}}.$$
Then $f$ is strongly starlike of order $\alpha$. Moreover, $f$ is strongly convex of order $((1-\gamma)\alpha+\delta)/\gamma.$
\end{corollary}
\begin{proof}
We obtain that $f$ is strongly starlike of order $\alpha$ if we take $h(z)=zf'(z)/f(z)$ and $\alpha=\beta$ in Theorem \ref{arg}. Since
\begin{align*}
   \bigg|\arg\left(1+\dfrac{zf''(z)}{f'(z)}\right)^\gamma\bigg|-\bigg|\arg\left(\dfrac{zf'(z)}{f(z)}\right)^{1-\gamma}\bigg|&\\
   \leq \bigg|\arg\left\{\left(\dfrac{zf'(z)}{f(z)}\right)^{1-\gamma}\left(1+\dfrac{zf''(z)}{f'(z)}\right)^\gamma\right\}\bigg|<\dfrac{\delta\pi}{2},
\end{align*}then by Theorem \ref{arg} we have
\begin{align*}
\bigg|\arg\left(1+\dfrac{zf''(z)}{f'(z)}\right)^\gamma\bigg|&
\leq\bigg|\arg\left(\dfrac{zf'(z)}{f(z)}\right)^{1-\gamma}\bigg|+\dfrac{\delta\pi}{2}\\
&<\dfrac{(1-\gamma)\alpha\pi}{2}+\dfrac{\delta\pi}{2}.
\end{align*}Hence $f$ is strongly convex of order $((1-\gamma)\alpha+\delta)/\gamma.$
\end{proof}



\begin{thebibliography}{99}
\bibitem{brki}D. A. Brannan\ and\ W. E. Kirwan, On some classes of bounded univalent functions. J. Lond. Math. Soc. {\bf2}(1) (1969), 431--443.
\bibitem{bucka} C. Bucka\ and\ K. Ciozda, On a new subclass of the class S, Ann. Pol. Math. {\bf28}  (1973), 153--161.
\bibitem{dziok}J. Dziok, Applications of the Jack lemma,  Acta Math. Hungar. {\bf 105} (2004), 93--102.
\bibitem{ebadian}A. Ebadian\ and\ J. Sok\'{o}\l, On the subordination and superordination of strongly starlike functions, Math. Slovaca {\bf 66} (2016), no.~4, 815--822.
\bibitem{siva} S. S. Kumar\ and\ P. Yadav,  On Oblique Domains of Janowski Function, Math. Slovaca {\bf73} (2023), no.~2.
\bibitem{kuroki} K. Kuroki, S. Owa, Notes on the open door lemma. Rend. Semin. Mat. (Torino) {\bf70} (2012), 423--434.
\bibitem{sugawa} M. Li, T. Sugawa, Some extensions of the open door lemma. Stud. Univ. Babe\c{s}-Bolyai, Math. {\bf60}(3) (2015), 421--430.
\bibitem{miller}S. S. Miller\ and\ P. T. Mocanu, {\it Differential subordinations}, Monographs and Textbooks in Pure and Applied Mathematics, {\bf225} (2000), Marcel Dekker, Inc., New York.


\bibitem{nuno} M. Nunokawa\ and\ J. Sok\'{o}\l, Certain results for the class of Carath\'{e}odory functions, Mediterr. J. Math. {\bf 17} (2020), no.~1, Paper No. 11, 9 pp.
\bibitem{nuno2}M. Nunokawa\ and\ J. Sok\'{o}\l, Strongly gamma-starlike functions of order alpha, Ann. Univ. Mariae Curie-Sk\l odowska Sect. A {\bf 67} (2013), no.~2, 43--51.
\bibitem{nunohun}M. Nunokawa, J. Sok\'{o}\l \ and\ K. T. Wi\c{e}c\l aw, On the order of Strongly starlikeness of some classes of starlike functions, Acta Math. Hungar. {\bf 145} (2015), no.~1, 142--149.
\bibitem{patel} J. Patel, On certain subclass of $p$-valently Bazilevic functions, JIPAM. J. Inequal. Pure Appl. Math. {\bf 6} (2005), no.~1, Article 16, 13 pp.
\bibitem{sokol} J. Sok\'{o}\l\ and\  L. Trojnar-Spelina, On a sufficient condition for strongly starlikeness, J. Inequal. Appl. {\bf383}, 11 pp (2013).
\bibitem{stan}J. Stankiewicz, On a family of starlike function,. Ann. Univ. Mariae Curie-Sk\l odowska, Sect. A {\bf22--24} (1968/70), 175--181.

\bibitem{taka} N. Takahashi\ and\ M. Nunokawa, A certain connection between starlike and convex functions. Appl. Math. Lett. {\bf16}, 653--655 (2003).
\end{thebibliography}
\end{document}